\documentclass[reqno]{amsart}
\usepackage{graphicx}
\usepackage{amscd,amssymb,amsthm}
\newtheorem{theorem}{Theorem}[section]

\newtheorem{lemma}[theorem]{Lemma}
\newtheorem{definition}[theorem]{Definition}
\newtheorem{remark}[theorem]{Remark}
\numberwithin{equation}{section}

\begin{document}

\title[Some Inequalities of differential polynomials]{Some Inequalities of differential polynomials II}
\author[J.F. Xu H.X. Yi and Z.L. Zhang]
{Junfeng Xu, Hongxun Yi and Zhanliang Zhang}

\address{Department of Mathematics, Wuyi University, Jiangmen 529020, P.R.China}
\email{xujunf@gmail.com}

\address{Department of Mathematics, Shandong University, Jinan
 250100, Shandong, P.R.China} \email{hxyi@sdu.edu.cn}

\address{Department of Mathematics, Zhaoqing University, Zhaoqing 526061, P.R.China}
\email{zlzhang@zqu.edu.cn}

\keywords{meromorphic function; differential polynomials; Nevanlinna
theory; value distribution} \subjclass[2000]{30D35; 26D10}

\begin{abstract}
In this paper, we consider the value distribution of the
differential polynomials $f^2f^{(k)}-1$ where $k$ is a positive
integer, and obtain some estimates only by the reduced counting
function. Our result answers a question in (Some inequalities of
differential polynomials, Mathematical Inequalities and
Applications, \textbf{12}(2009), no.1, 99--113) completely.
\end{abstract}

\maketitle
\section{Introduction and results}\label{sec1}

~Let $\mathbb{C}$ be the open complex plane and $\mathcal{D}\in
\mathbb{C}$ be a domain. Let $f$ be a meromorphic function in the
complex plane, we assumed that the reader is familiar with the
notations of Nevanlinna theory(see, e.g.,\cite{Hayman,YY,Yang}).

\begin{definition} Let $k$ be a positive integer, for any constant $a$ in the
complex plane. We denote by $N_{k)}(r,1/(f-a))$ the counting
function of $a$-points of $f$ with multiplicity $\leq k$, by
$N_{(k}(r,1/(f-a))$ the counting function of $a$-points of $f$ with
multiplicity $\geq k$, by $N_{k}(r,1/(f-a))$ the counting function
of $a$-points of $f$ with multiplicity of $k$. and denote the
reduced counting function by  $ \overline{N}_{k)}(r,1/(f-a))$,
$\overline{N}_{(k}(r,1/(f-a))$ and $\overline{N}_{k}(r,1/(f-a))$,
respectively.
\end{definition}

Recently, Huang and Gu (\cite{HG}) have obtained a quantitative
result about a differential polynomials $f^2f^{(k)}-1$. They proved
the following theorem.

\textbf{Theorem A.}\quad \emph{Let $f$ be transcendental meromorphic
in the complex plane and $k$ be a positive integer, then
\begin{equation}\label{a0}
   T(r,f)\leq 6N(r,\frac{1}{f^2f^{(k)}-1})+S(r,f).
\end{equation}
}

\begin{remark}
In fact, Q. Zhang \cite{Zhangq} proved the case of $k=1$. X. Huang
and Y. Gu proved the case of $k\geq 2$.
\end{remark}

As we all known, the second fundamental theorem in Nevanlinna's
theory of value distribution use the reduced counting function to
estimate the Nevanlinna characteristic function(cf.\ \cite{Ya}).
Naturally, we can pose the following important question.

Whether one can give some quantitative estimates on the generally
differential polynomials by the reduced counting function?

 In \cite{XYZ}, the authors give some affirmative answers.

\textbf{Theorem B.}\quad \emph{Let $f$ be a transcendental
meromorphic function,
$L[f]=a_{k}f^{(k)}+a_{k-2}f^{(k-2)}+\cdots+a_0f$, where
$a_0,a_1,\cdots,a_{k}(\not\equiv0)$ are small functions, for $c
(\neq0,\infty)$, let $F=f^2L[f]-c$,  then there exists a constant
$M>0$, which does not depend on $f$, such that
$$T(r,f)\leq M\,\overline{N}(r,\frac{1}{F})+S(r,f).$$}

\begin{remark} We know $F$ has infinitely many zeros, and the constant $M$ is at
least 6 from Theorem A. But the method of Theorem B can't give the
certain coefficient. Hence, we want to get the more precise estimate
of the coefficient. In fact, we proved the following result in
\cite{XYZ} by giving some restriction on the zeros of $f$.
\end{remark}

\textbf{Theorem C.}\quad \emph{ Let $f$ be a transcendental
meromorphic function, and let $k$ be a positive integer. If
$N_{1}(r,\frac{1}{f})=S(r,f)$, then
\begin{equation}\label{a1}
T(r,f)\leq 2\overline{N}(r,\frac{1}{f^2f^{(k)}-1})+S(r,f).
\end{equation}
}

In the paper, we continue to investigate the problem in this
direction.  Though Theorem C has the smaller coefficient 2, we know
the condition of the simple zero is not necessary from Theorem B.
Hence it is an important question how to remove the condition and
get a precise estimation. We prove the following theorem.

\begin{theorem}
Let $f$ be a transcendental meromorphic function, and let $k$ be a
positive integer. Then
\begin{equation}\label{a2}
   T(r,f)\leq M\,\overline{N}(r,\frac{1}{f^2f^{(k)}-1})+S(r,f).
\end{equation}
where $M$ is 6 if $k=1$ or $k\geq 3$,  $M$ is 10 if $k=2$.
\end{theorem}

\section{Proof of the theorem}

In order to prove our result, we need to the following lemmas.

\begin{lemma}\label{lm4}
Let $f$ be transcendental meromorphic function, and let $k$ be a
positive integer. Then
\begin{equation}\label{b1}
\begin{array}{lll}
3 T(r,f)&\leq&\displaystyle \overline{N}(r,f)+\overline{N}(r,\frac{1}{f})+N_{k)}(r,\frac{1}{f})+k\overline{N}_{(k+1}(r,\frac{1}{f})\\
&+&\displaystyle \overline{N}(r,\frac{1}{f^2f^{(k)}-1})-N_0(r,
\frac{1}{(f^2f^{(k)})'})+S(r,f).
\end{array}
\end{equation}
where $N_0(r, \frac{1}{(f^2f^{(k)})'})$ denotes the counting
function of the zeros of $(f^2f^{(k)})'$, not of $f(f^2f^{(k)}-1)$.
Especially, if $k=1$, we get
\begin{equation}\label{b1a}
\begin{array}{lll}
\displaystyle  3 T(r,f)&\leq& \displaystyle \overline{N}(r,f)+2\overline{N}(r,\frac{1}{f})\\
&+&\displaystyle \overline{N}(r,\frac{1}{f^2f'-1})-N_0(r,
\frac{1}{(f^2f')'})+S(r,f).
\end{array}
\end{equation}
\end{lemma}
\begin{proof}
We first claim that $f^2f^{(k)}\not\equiv$ constant. If
$f^2f^{(k)}\not\equiv C$, where $C$ is a constant.

Obviously, $C\neq 0$. Hence $f$ has no zero and
$\frac{1}{f^3}=\frac{1}{C}\frac{f^{(k)}}{f}$. Therefore,
\begin{equation*}
\begin{array}{lll}
3T(r,f)&=& \displaystyle m(r,\frac{1}{f^3})+N(r,\frac{1}{f^3})+O(1)\\
&=& \displaystyle m(r,\frac{f^{(k)}}{f})+O(1)=S(r,f).
\end{array}
\end{equation*}
a contradiction. Hence $f^2f^{(k)}$ is not equivalent to a constant.

Let
$$
\frac{1}{f^3}\equiv
\frac{f^2f^{(k)}}{f^3}-\frac{(f^2f^{(k)})'}{f^3}\frac{f^2f^{(k)}-1}{(f^2f^{(k)})'},
$$
we have
\begin{equation*}
\begin{array}{llll}
3m(r,\frac{1}{f})&=&  \displaystyle m(r,\frac{1}{f^3})\\[.3cm]
&\leq& \displaystyle m(r,\frac{f^2f^{(k)}-1}{(f^2f^{(k)})'})
+m(r,\frac{f^{(k)}}{f})+m(r,\frac{(f^2f^{(k)})'}{f^3})+O(1)\\[.3cm]
&\leq & \displaystyle
N(r,\frac{(f^2f^{(k)})'}{f^2f^{(k)}-1})-N(r,\frac{f^2f^{(k)}-1}{(f^2f^{(k)})'})+S(r,f)\\[.3cm]
&=& \displaystyle
N(r,(f^2f^{(k)})')+N(r,\frac{1}{f^2f^{(k)}-1})-N(r,\frac{1}{(f^2f^{(k)})'})-\\[.3cm]
&&\displaystyle
N(r,f^2f^{(k)}-1)+S(r,f)\\[.3cm]
&=& \displaystyle
\overline{N}(r,f)+N(r,\frac{1}{f^2f^{(k)}-1})-N(r,\frac{1}{(f^2f^{(k)})'})+S(r,f).
\end{array}
\end{equation*}
Hence
\begin{equation}\label{b2}
\begin{array}{lll}
3T(r,f)&=& \displaystyle 3 m(r,\frac{1}{f})+ 3N(r,\frac{1}{f})+O(1)\\
&=& \displaystyle \overline{N}(r,f)+ 3 N(r,\frac{1}{f})+
N(r,\frac{1}{f^2f^{(k)}-1})-N(r,\frac{1}{(f^2f^{(k)})'})+S(r,f).
\end{array}
\end{equation}
Let
\begin{equation}\label{b3}
 N(r,\frac{1}{(f^2f^{(k)})'})=N_{000}(r,\frac{1}{(f^2f^{(k)})'})+N_{00}(r,\frac{1}{(f^2f^{(k)})'})+N_0(r,\frac{1}{(f^2f^{(k)})'})
\end{equation}
where $N_{000}(r, \frac{1}{(f^2f^{(k)})'})$ denotes the counting
function of the zeros of $(f^2f^{(k)}-1)'$, which come from the
zeros of $f^2f^{(k)}-1$, $N_{00}(r, \frac{1}{(f^2f^{(k)})'})$
denotes the counting function of the zeros of $(f^2f^{(k)}-1)'$,
which come from the zeros of $f$. Hence we have
\begin{equation}\label{b4}
 N(r,\frac{1}{f^2f^{(k)}-1})-N_{000}(r,\frac{1}{(f^2f^{(k)})'})=
 \overline{N}(r,\frac{1}{f^2f^{(k)}-1}).
\end{equation}

Supposed that $z_0$ is a zero of $f$ with multiplicity $q$, if
$q\leq k$, then $z_0$ is a zero of $(f^2f^{(k)})'$ with multiplicity
at least $2q-1$; if $q\geq k+1$, then $z_0$ is a zero of
$(f^2f^{(k)})'$ with multiplicity at least $3q-(k+1)$. Hence we have
\begin{equation}\label{b5}
\begin{array}{lll}
\displaystyle
3N(r,\frac{1}{f})-N_{00}(r,\frac{1}{(f^2f^{(k)})'})&\leq &
\displaystyle
N_{k)}(r,\frac{1}{f})+\overline{N}_{k)}(r,\frac{1}{f})+(k+1)\overline{N}_{(k+1}(r,\frac{1}{f})\\[.3cm]
&=& \displaystyle
N_{k)}(r,\frac{1}{f})+\overline{N}(r,\frac{1}{f})+k\overline{N}_{(k+1}(r,\frac{1}{f}).
\end{array}
\end{equation}

Combining (\ref{b2})-(\ref{b5}), we have
\begin{equation}\nonumber
\begin{array}{lll}
3 T(r,f)&\leq&\displaystyle \overline{N}(r,f)+\overline{N}(r,\frac{1}{f})+N_{k)}(r,\frac{1}{f})+k\overline{N}_{(k+1}(r,\frac{1}{f})\\
&+&\displaystyle \overline{N}(r,\frac{1}{f^2f^{(k)}-1})-N_0(r,
\frac{1}{(f^2f^{(k)})'})+S(r,f).
\end{array}
\end{equation}
This completes the proof of the lemma.
\end{proof}

\begin{lemma}\label{lm3}\cite{Zhangq}
Let $f$ be a transcendental meromorphic function, and let $k$ be a
positive integer. Let
\begin{equation}\label{b6}
G(z)=13(\frac{F'}{F})^2+20(\frac{F'}{F})'-24\frac{F'}{F}\frac{l'}{l}+
    8(\frac{l'}{l})^2-88(\frac{l'}{l})',
\end{equation}
then we have (1) $G(z)\not\equiv 0$; (2) The simple poles of $f(z)$
are the zeros of $G(z)$.
\end{lemma}

Now we begin to prove Theorem 1.4.

$(I)$ If $k=1$, Q. D. Zhang proved the inequality $(\ref{a0})$ by
using the auxiliary function. Here we use his method to construct
the function $G(z)$, and can obtain a better result if $k=1$.

Let $F(z)=f^2f'-1$ and $l(z)=\frac{F'}{f}=2(f')^2+ff''$. Obviously
$l(z)\not\equiv0$. Also let
\begin{equation}\label{c3}
    G(z)=13(\frac{F'}{F})^2+20(\frac{F'}{F})'-24\frac{F'}{F}\frac{l'}{l}+
    8(\frac{l'}{l})^2-88(\frac{l'}{l})'.
\end{equation}

By Lemma \ref{lm3}, we know $G(z)\not\equiv0$, and the simple poles
of $f$ are the zeros of $G(z)$. Note that the poles of $G(z)$ whose
multiplicity is at most two come from the multiple poles of $f$, $F$
or the zeros of $l$. But it is still difficult to deal with the
zeros of $l$. We consider the poles of $\beta^2 G(z)$. By
differentiating the equation $F(z)=f^2f'-1$, we get
 \begin{equation}\label{m5}
 f\beta=-\frac{F'}{F},
\end{equation}
where
\begin{equation}\label{c4}
\beta=2(f')^2+ff''-ff'\frac{F'}{F},\quad  l=-\beta F.
\end{equation}

We can see the zeros of $l$ either is the zeros of $F$, or the zeros
of $\beta$. From the above we know that the multiple poles of $f$
with the multiplicity $q(\geq 2)$ is the zeros of $\beta$ with the
multiplicity of $q-1$. Hence the poles of $\beta^2 G(z)$ only come
from the zeros of $F$, and the multiplicity is at most  4. Hence,
$$
N(r, \beta^2 G)\leq 4 \overline{N}(r, 1/F).
$$

Note that
$m(r,G)=S(r,f)$, therefore $m(r, \beta^2 G)=S(r,f)$. Hence
$$
T(r, \beta^2G)\leq 4\overline{N}(r, 1/F).
$$

Since the multiple zeros of $f$ with the multiplicity $p (\geq 2)$
are the multiple zeros of $\beta$ with multiplicity at least $2p-2$,
therefore, are at least the zeros of $\beta^2 G$ with the
multiplicity $2(2p-2)-2=4p-6$. Also note that the simple poles of
$f$ are the zeros of $\beta^2 G$. Hence we have
\begin{equation}\label{c5}
\overline{N}_{1)}(r,f) + 2
N(r,\frac{1}{f})-2\overline{N}(r,\frac{1}{f})\leq N(r,
\frac{1}{\beta^2 G})\leq T(r, \beta^2 G)\leq 4\overline{N}(r,
\frac{1}{F}).
\end{equation}

Combining (\ref{b1a}) and (\ref{c5}), we have
\begin{equation}\nonumber
\begin{array}{lll}
&& \displaystyle T(r,f) + N_{(2}(r,f)-2\overline{N}_{(2}(r,f)+m (r,
f)+ 4m(r,\frac{1}{f})+6N(r,\frac{1}{f})-6\overline{N}(r,\frac{1}{f}) \\[.3cm]
&\leq & \displaystyle 6\overline{N}(r,\frac{1}{f^2f'-1})+S(r,f),
\end{array}
\end{equation}

Hence we have
\begin{equation}\label{c6}
T(r,f)<6\overline{N}(r,\frac{1}{f^2f'-1})+S(r,f).
\end{equation}
Obviously, our result improves the conclusion of Q.D. Zhang
greatly.\\

$(II)$ If $k\geq 2$, X. J. Huang and Y. X. Gu constructed the
similar function $G_1(z)$.

Let $F_1(z)=f^2f^{(k)}-1$ and $l_1(z)=\frac{F_1'}{f}=2(f')^2+ff''$.
Obviously $l_1(z)\not\equiv0$. Let
\begin{equation}\label{c3}
    G_1(z)=a_1(\frac{F_1'}{F_1})^2+a_2(\frac{F_1'}{F_1})'+a_3\frac{F_1'}{F_1}\frac{l_1'}{l_1}+
    a_4(\frac{l_1'}{l_1})^2+a_5(\frac{l_1'}{l_1})'.
\end{equation}
Where
\begin{eqnarray*}
  a_1 &=& 2(k+1)^2-\frac{(3k+7)(k^2-4k-29)}{(k+3)} \\
  a_2 &=& -(k+5)(k^2-4k-29); \\
  a_3 &=& 4(k^2-4k-29); \\
  a_4 &=& -4(k+3)(k+1)\\
  a_5 &=& 2(k+2)(k+3)(k+5).
\end{eqnarray*}

By Lemma 3 in \cite{HG}, we know $G_1(z)\not\equiv0$, and Lemma 4 of
\cite{HG}, we know the simple poles of $f$ are the zeros of
$G_1(z)$. Note the poles of $G_1(z)$ come from the multiple poles of
$f$, $F_1$ or the zeros of $l_1$, whose multiple is at most two. But
it is also difficult to deal with the zeros of $l_1$.

We consider the poles of the function $\beta^2 G_1(z)$. Similar with
the proof of the (\ref{m5}),
\begin{equation}\nonumber
\beta=2(f')^2+ff''-ff'\frac{F_1'}{F_1},\quad  l_1=-\beta F_1.
\end{equation}
Then we can see the zero of $l_1$ either is the zero of $F_1$, or
the zero of $\beta$. From the above, we know the multiple zeros of
$f$ with the multiplicity $q(\geq 2)$ are the zeros of $\beta$ with
the multiplicity $q-1$. Hence the poles of $\beta^2 G_1(z)$ only
come from the zeros of $F$, and the multiplicity are at most four.
Therefore,
$$
N(r, \beta^2 G)\leq 4 \overline{N}(r, 1/F).
$$

Note that
$m(r, G)=S(r,f)$, therefore $m(r, \beta^2 G)=S(r,f)$. Hence
$$
T(r, \beta^2
G)\leq 4\overline{N}(r, 1/F).
$$

Then by the zeros of $f$ with multiplicity $p(\geq k)$ are at least
 the zeros of $\beta$ with the multiplicity $2p-2$, therefore are at
least the zeros of $\beta^2 G$ with the multiplicity
$2(2p-2)-2=4p-6$. Note that the simple poles of $f$ are also the
zeros of $\beta^2 G$. Hence we have
\begin{equation}\label{c4}
\overline{N}_{1)}(r,f) +
4N_{(k}(r,\frac{1}{f})-6\overline{N}_{(k}(r,\frac{1}{f})\leq N(r,
\frac{1}{\beta^2 G})\leq T(r, \beta^2 G)\leq 4\overline{N}(r,
\frac{1}{F}).
\end{equation}

Next we divide two cases:\\

Case (1). If $k\geq 3$, the we have
\begin{equation}\label{c5}
\overline{N}_{1)}(r,f) + 2 N_{(k}(r,\frac{1}{f})\leq
4\overline{N}(r, \frac{1}{F}).
\end{equation}

Combining the doubled (\ref{b1}) and (\ref{c5}), we have
\begin{equation}\label{c6}
\begin{aligned}
6 T(r,f) &+ \overline{N}_{1)}(r,f) + 2N_{(k}(r,\frac{1}{f})-
2\overline{N}(r,f)-2\overline{N}(r,\frac{1}{f})
 - 2N_{k)}(r,\frac{1}{f})-2k\overline{N}_{(k+1}(r,\frac{1}{f})\\
&\leq  6\overline{N}(r,\frac{1}{f^2f^{(k)}-1})-N_0(r,
\frac{1}{(f^2f^{(k)})'})+S(r,f).
\end{aligned}
\end{equation}

Then
\begin{equation}\label{c7}
\begin{aligned}
 6T(r,f) & + \overline{N}_{1)}(r,f) +2 N_{(k}(r,\frac{1}{f})-
2\overline{N}(r,f)-2\overline{N}(r,\frac{1}{f})
 - 2N_{k)}(r,\frac{1}{f})-2k\overline{N}_{(k+1}(r,\frac{1}{f})\\
&\geq T(r,f)+ m(r,f)+ N(r,f)+
\overline{N}_{1)}(r,f)-2\overline{N}(r,f)+4m(r,\frac{1}{f})+
4N(r,\frac{1}{f})\\
& + 2N_{(k}(r,\frac{1}{f})-2\overline{N}(r,\frac{1}{f})
 - 2N_{k)}(r,\frac{1}{f})-2k\overline{N}_{(k+1}(r,\frac{1}{f}) .
\end{aligned}
\end{equation}

In (\ref{c7}), we first consider the case of the pole
\begin{equation}\label{c8}
\begin{aligned}
     &\, N(r,f)+ \overline{N}_{1)}(r,f)-2\overline{N}(r,f)\\
\geq &\, N(r,f)+
\overline{N}_{1)}(r,f)-2\overline{N}_{1)}(r,f)-2\overline{N}_{(2}(r,f)\\
\geq &\, N(r,f)-\overline{N}_{1)}(r,f)-2\overline{N}_{(2}(r,f)\\
= &\,
N_{1)}(r,f)+N_{(2}(r,f)-\overline{N}_{1)}(r,f)-2\overline{N}_{(2}(r,f)\\
>&0.
\end{aligned}
\end{equation}

In (\ref{c7}), we consider the case of the zero
\begin{equation}\label{c9}
\begin{aligned}
      &\,4N(r,\frac{1}{f})+2N_{(k}(r,\frac{1}{f})-2\overline{N}(r,\frac{1}{f})
         -
         2N_{k)}(r,\frac{1}{f})-2k\overline{N}_{(k+1}(r,\frac{1}{f})\\
\geq\,&\,4N(r,\frac{1}{f})+2N_{k}(r,\frac{1}{f})+
N_{(k+1}(r,\frac{1}{f})-2\overline{N}(r,\frac{1}{f})
         -
         2N_{k)}(r,\frac{1}{f})-\frac{2k}{k+1}N_{(k+1}(r,\frac{1}{f})\\
\geq\,&\,4N(r,\frac{1}{f})+N_{k}(r,\frac{1}{f})+
\frac{2}{k+1}N_{(k+1}(r,\frac{1}{f})-2\overline{N}(r,\frac{1}{f})
         -2N_{k)}(r,\frac{1}{f})\\
         >& \, 0.
\end{aligned}
\end{equation}
From (\ref{c6})-(\ref{c9}),  we have
\begin{equation}\label{c10}
T(r,f)<6\overline{N}(r,\frac{1}{f^2f^{(k)}-1})+S(r,f).
\end{equation}\\

Case (2). If $k= 2$, by (\ref{c4}), we have
\begin{equation}\nonumber
\overline{N}_{1)}(r,f) + N_{(k}(r,\frac{1}{f})\leq 4\overline{N}(r,
\frac{1}{F}).
\end{equation}
Similar with the above discussion, we have
\begin{equation}\label{c11}
T(r,f)<10\overline{N}(r,\frac{1}{f^2f^{''}-1})+S(r,f).
\end{equation}

This completes the proof of Theorem 1.4.

\bigskip
{\bf Acknowledgement}
Supported by NNSF of China(10771121), NSFC-RFBR and
Fund of Education Department of Guangdong Province(LYM08097).

\end{document}